\documentclass{amsart}

\usepackage{graphicx}

\newcommand{\C}{\mathbb{C}}
\newcommand{\T}{\mathbb{T}}
\newcommand{\D}{\mathbb{D}}

\newtheorem{theorem}{Theorem}

\newtheorem{corollary}[theorem]{Corollary}
\theoremstyle{definition}
\newtheorem{remark}[theorem]{Remark}

\DeclareMathOperator{\Lip}{Lip}
\DeclareMathOperator{\Ind}{ind}
\DeclareMathOperator{\wind}{wind}

\begin{document}
\title[Toeplitz operators on $H^1$]{Toeplitz operators with piecewise continuous symbols on the Hardy space $H^1$} 
\author[Miihkinen and Virtanen]{Santeri Miihkinen and Jani A.~Virtanen}

\address{{Department of Mathematics, \AA}bo Akademi University, Tuomiokirkontori~3, Turku, Finland}
\email{santeri.miihkinen@abo.fi}

\address{Department of Mathematics and Statistics, University of Reading, Reading RG6~6AX, England}
\email{j.a.virtanen@reading.ac.uk}



\thanks{J. Virtanen was supported in part by Engineering and Physical Sciences Research Council grant EP/M024784/1}

\maketitle

\begin{abstract}
The geometric descriptions of the (essential) spectra of Toeplitz operators with piecewise continuous symbols are among the most beautiful results about Toeplitz operators on Hardy spaces $H^p$ with $1<p<\infty$. In the Hardy space $H^1$, the essential spectra of Toeplitz operators are known for continuous symbols and symbols in the Douglas algebra $C+H^\infty$. It is natural to ask whether the theory for piecewise continuous symbols can also be extended to $H^1$. We answer this question in the negative and show in particular that the Toeplitz operator is never bounded on $H^1$ if its symbol has a jump discontinuity. 
\end{abstract}

\section{Introduction}
For $1\leq p\leq \infty$, the {\it Hardy space} $H^p$ of the unit circle $\T$ is defined by
\begin{equation}
	H^p = \{ f\in L^p : f_k = 0\ {\rm for}\ k<0\},
\end{equation}
where $f_k$ stands for the $k$th Fourier coefficient of $f$, and the {\it orthogonal (Riesz) projection} of $L^2$ onto $H^2$ is denoted by $P$. The Riesz projection $P$ can be expressed as a singular integral operator as follows:
\begin{equation}
	P=\frac12(I+S),
\end{equation}
where $S$ is the Cauchy singular integral operator defined by
\begin{equation}
	Sf(t) = \frac1{\pi i} \int_{\T} \frac{f(\tau)}{\tau - t} \qquad(t\in\T),
\end{equation}
which is understood in the Cauchy principal value sense. It is well known that the Cauchy singular integral operator is bounded on $L^p$ for $1<p<\infty$ (see, e.g.,~\cite{BK} or~\cite{BS}) and hence $P$ is a bounded linear operator of $L^p$ onto $H^p$ for $1<p<\infty$. When $p=1$ or $p=\infty$, both operators $S$ and $P$ are unbounded. 

For $a\in L^\infty$, we define the {\it Toeplitz operator} $T_a$ on $H^p$ with {\it symbol} $a$ by
\begin{equation}
	T_a f = P(af).
\end{equation}
It is clear that $T_a$ is bounded on $H^p$ if $1<p<\infty$. It is also known that there are no unbounded symbols that generate bounded Toeplitz operators on $H^p$, and that nontrivial Toeplitz operators cannot be compact. The question of boundedness in $H^1$ is much more interesting. Indeed, the Toeplitz operator $T_a$ is bounded on $H^1$ if and only if
$a\in L^\infty$ and $Qa$ is of logarithmic bounded mean oscillation (where $Q = I-P$ is the complementary projection). This result has been proved by several authors---most recently in~\cite{PapV}. 

Much less is known about their spectral properties  when $p=1$, and in particular, the Fredholm properties of Toeplitz operators are only understood for certain continuous symbols and symbols in the Douglas algebra. 

In the next section, we recall some basic theory on (logarithmic) mean oscillation and then present a well-known geometric description of the essential spectra of Toeplitz operators $T_a : H^p\to H^p$ with piecewise continuous symbols $a$, which can be obtained as the union of the limit values of $a$ and certain $p$-circular arcs joining the jumps of the symbol $a$. In the last section, we consider the same problem for Toeplitz operators on the Hardy space $H^1$. It may be tempting to ask what should replace the $p$-circular arcs in this case. Yet, perhaps surprisingly, we show that $T_a$ is never bounded on $H^1$ if $a$ possesses a jump and therefore the question has no meaning in the world of bounded linear operators. We do this by combining a little known result of Lindel\"of on limit values of functions in $H^\infty$ and a decomposition of the symbol class that generate bounded Toeplitz operators on $H^1$.

\section{Logarithmically weighted bounded mean oscillation}
We say that a function $f\in L^1$ is of \emph{bounded mean oscillation} and write $f\in BMO$ if
\begin{equation}\label{e:BMO}
	\sup_{I} \frac1{|I|} \int_I |f-f_I|<\infty
\end{equation}
over all subarcs $I$ of $\T$, where $f_I = |I|^{-1} \int_I f$. If, in addition,
\begin{equation}\label{e:VMO}
	\lim_{\delta \to 0} \sup_{|I|<\delta} \frac1{|I|} \int_I |f-f_I| = 0,
\end{equation}
we say that $f$ is of \emph{vanishing mean oscillation} and write $f\in VMO$. Note that
\begin{equation}
	L^\infty \subset BMO \subset \bigcap_{1\leq p<\infty} L^p,\quad C\subset VMO \subset BMO,
\end{equation}
where $C$ is the space of all continuous functions on $\T$ (see, e.g., 1.48 of~\cite{BS}). These spaces can be used to characterize bounded and compact Hankel operators $H_a$ on $H^p$; that is, $H_a$ is bounded if and only if $Pa\in BMO$; while $H_a$ is compact if and only if $Pa\in VMO$.

In the context of $H^1$, we equip the two spaces with logarithmic weights as follows. We say that a function $f\in L^1$ is of \emph{logarithmic bounded mean oscillation} and write $f\in BMO_{\log}$ if
\begin{equation}\label{e:BMOlog}
	\sup_{I} \frac{\log\frac{4\pi}{|I|}}{|I|} \int_{\T} | f- f_I | < \infty.
\end{equation}
If, in addition, 
\begin{equation}\label{e:VMOlog}
	\lim_{\delta \to 0} \sup_{|I|<\delta} \frac{\log\frac{4\pi}{|I|}}{|I|} \int_I |f-f_I| = 0,
\end{equation}
we say that $f$ is of \emph{logarithmic vanishing mean oscillation} and write $f\in VMO_{\log}$. Characterizations of bounded and compact Hankel operators on $H^1$ can be given in terms of these spaces analogously to the case $1<p<\infty$  (see~\cite{PapV}). For bounded Toeplitz operators on $H^1$, we have the following characterization (see~\cite{PapV} and the references therein):

\begin{theorem}\label{boundedness}
A Toeplitz operator $T_a$ is bounded on $H^1$ if and only if $a\in L^\infty$ and $Qa\in BMO_{\log}$.
\end{theorem}

When dealing with Toeplitz operators on $H^1$ with piecewise continuous symbols, we need the following description of $BMO_{\log}$ due to Janson~\cite{Jan}:
\begin{equation}\label{e:Janson}
	BMO_{\log} = \{ f + Pg : f,g\in \Lip_{\log} \},
\end{equation}
where $\Lip_{\log}$ is the logarithmic Lipschitz space defined by
\begin{equation}
	\Lip_{\log} = \left\{ f\in C : \sup_{w,z\in\T} \log (4|w-z|^{-1}) |f(w)-f(z)| < \infty\right\}.
\end{equation}
A simple consequence of this result is the following decomposition:
\begin{equation}\label{e:decomposition}
	\{ f\in L^\infty : Qf\in BMO_{\log} \} = \Lip_{\log} + H^\infty.
\end{equation}
Indeed, if $a = l+h \in \Lip_{\log}+H^\infty$, then clearly $a\in L^\infty$ and $Qa = Ql\in BMO_{\log}$ according to \eqref{e:Janson}. Conversely, if $a\in L^\infty$ and $Qa\in BMO_{\log}$, then $Qa = f + Pg$ for some $f,g\in \Lip_{\log}$, and so $Qa=Qf$. Thus, $a-f\in H^\infty$ and so we have $a = f+(a-f) \in \Lip_{\log} + H^\infty$.

\section{Piecewise continuous symbols}

One of most beautiful results about Toeplitz operators is the geometric description of their (essential) spectra for piecewise continuous symbols; see~\cite{BK, GK}, which includes the most general case of Hardy spaces $H^p(\mu, \Gamma)$ with Muckenhoupt weights $\mu$, Carleson curves $\Gamma$ and $1<p<\infty$. Here we recall the result for unweighted Hardy spaces of the unit circle.

We say that a bounded linear operator $T$ on a Banach space $X$ is a \emph{Fredholm} operator of \emph{index} $\kappa$ if both $\dim\ker T$ and $\dim X/T(X)$ are finite, and $\kappa = \Ind T = \dim\ker T - \dim X/T(X)$. The essential spectrum $\sigma_{\rm ess}(T)$ of $T$ is defined by
$$
	\sigma_{\rm ess}(T) = \{\lambda\in\C : T-\lambda\ \textrm{is not Fredholm}\}.
$$
For $a\in L^\infty$ and $1<p<\infty$, it is well known that $T_a$ is invertible on $H^p$ if and only if $T_a$ is Fredholm of index zero. If $a$ is piecewise continuous, then
\begin{equation}\label{e:PC}
	\sigma_{\rm ess} (T_a) =
\left(\bigcup_{t \in \mathbb{T}} \{a(t \pm 0)\}\right)
\bigcup
\left(\bigcup_{a(t - 0) \not= a(t + 0)}
\text{\rm Arc}_p(a; t)\right) , 
\end{equation}
where
$$
\text{\rm Arc}_p(a; t) :=
\left\{\zeta \in \mathbb{C} : 
\arg\frac{a(t - 0) - \zeta}{a(t + 0) - \zeta} = \frac{2\pi}p\right\}
$$
is the $p$-circular arc consisting of the points from which the chord  $[a(t + 0), a(t - 0)]$, $a(t - 0) \not= a(t + 0)$ is seen at the angle $2\pi/p$; see Figure~\ref{fig}.

\begin{figure}
\caption{Examples of $Arc_p(a; t)$ for values $p = 3$ (left) and $p = \frac{3}{2}$ (right).}\label{fig}
\includegraphics[width=1\textwidth, trim={0 13cm 0 3cm},clip]{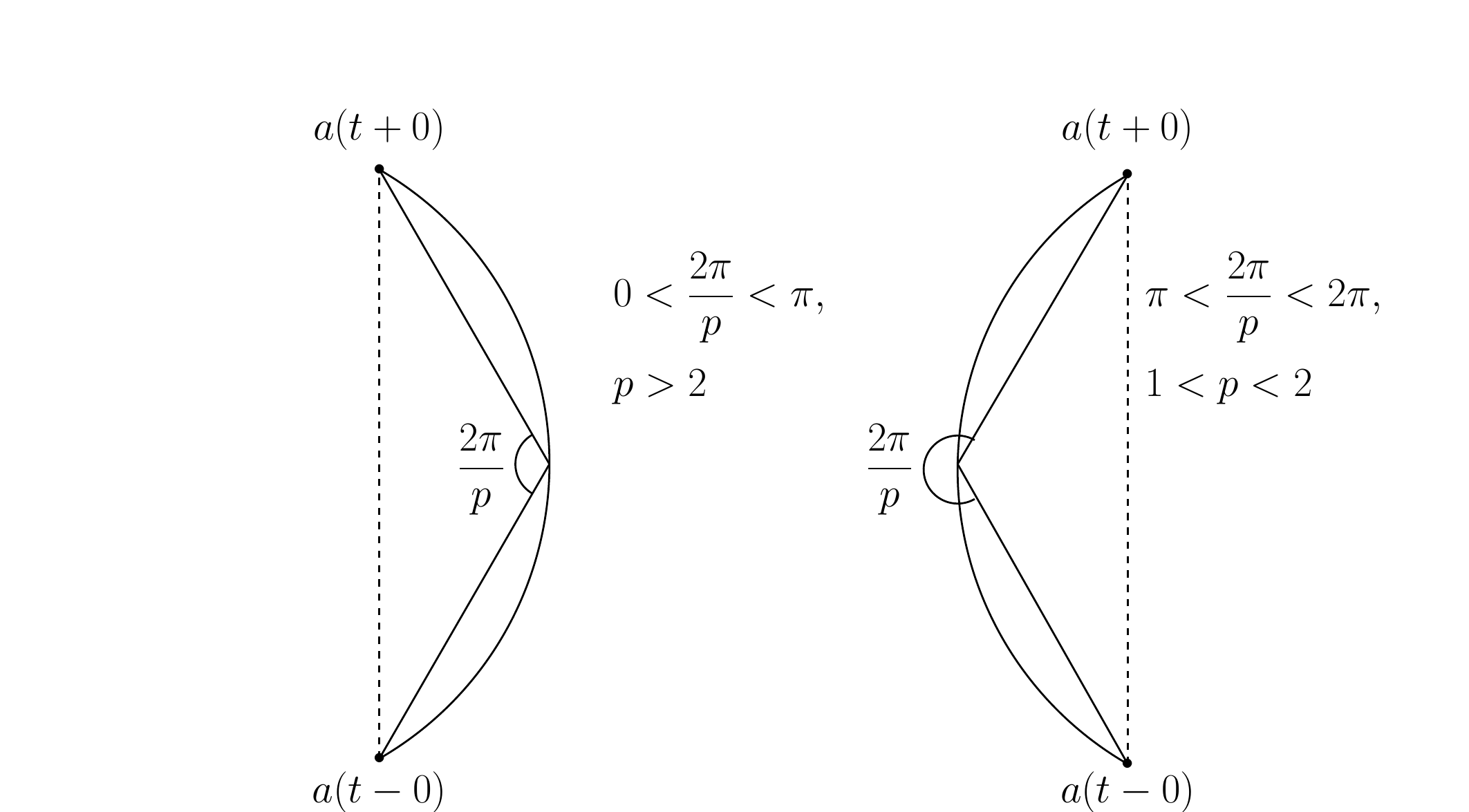} 
\end{figure}

We now turn our attention to the Fredholm properties of Toeplitz operators on $H^1$. The following result of~\cite{V} deals with continuous symbols.

\begin{theorem}\label{continuous symbols}
Let $a \in C \cap VMO_{\log}.$ Then $T_a$ is Fredholm if and only if $a$ is nowhere zero, in which case $\Ind T_a = -\wind a$.
\end{theorem}

The above theorem readily implies the following characterization of the essential spectrum of $T_a$.

\begin{corollary}
If $a \in C \cap VMO_{\log},$ then $\sigma_{\rm ess}(T_a) = a(\T).$
\end{corollary} 

This, however, is not the largest set of continuous symbols that generate bounded Toeplitz operators on $H^1$. Indeed, if $a\in C$ and $Qa\in BMO_{\log}$, then the Fredholm properties of $T_a$ are not known. We conjecture that still $\sigma_{\rm ess} (T_a) = a(\T)$ and the index formula remains the same. It is worth noting that the approach of Theorem~\ref{continuous symbols} is not applicable because it relies on the compactness of Hankel operators, which requires that $Pa\in VMO_{\log}$.

Another class of symbols for which the Fredholm properties are understood in $H^1$ is a certain Douglas-type algebra (see~\cite{PapV}):

\begin{theorem}
Let $a\in C\cap VMO_{\log} + \overline{H^\infty}\cap BMO_{\log}$. Then $T_a$ is Fredholm on $H^1$ if and only if there are $\epsilon>0$ and $\delta>0$ such that $|a(z)| \geq \epsilon$ whenever $1-\delta<|z|<1$, in which case $\Ind T_a = -\wind a_r$, where $a_r(t) = a(rt)$ with $1-\delta<r<1$ and $a(z)$ is the harmonic extension of $a$ for $z\in\D$. In particular, we get the following easy consequence:
$$
	\sigma_{\rm ess} (T_a) = \bigcap_{0<r<1} \overline{a(\D\setminus r\D)}.
$$
\end{theorem}

In connection with the previous theorem, it is worth noting the following result of~\cite{PapV}:
$$
	C\cap VMO_{\log} + \overline{H^\infty}\cap BMO_{\log} = \{ f\in L^\infty\cap BMO_{\log} : H_f\ \textrm{is compact on}\ H^1\},
$$
which ``almost'' implies that in order to deal with other classes of discontinuous symbols, we need to make do without compact Hankel operators. This leads us to piecewise continuous symbols.

We first recall the following little known result of Lindel\"of~\cite{Lin}, which shows that functions in $H^\infty$ cannot have discontinuities of the first kind; see also Exercise 7 of Chapter~II of~\cite{Gar}. 

\begin{theorem}
Functions in $H^\infty$ cannot have jumps, that is, if $f \in H^\infty$ and the one-sided limits $f(t\pm 0)$ exist at some $t \in \T$, then $f(t+0) = f(t-0)$.
\end{theorem}

We use Lindel\"of's result to prove the following theorem.

\begin{theorem}\label{main}
If $a$ is the symbol of a bounded Toeplitz operator on $H^1$, then $a$ cannot have jump discontinuities.
\end{theorem}
\begin{proof}
Suppose that $T_a$ is bounded on $H^1$. By Theorem~\ref{boundedness}, the symbol $a$ is in $L^\infty$ and $Qa$ is in $BMO_{\log}$, and so $a = f + g$ for some $f\in \Lip_{\log}$ and $g\in H^\infty$ according to~\eqref{e:decomposition}. If $a(t-0)\neq a(t+0)$ for some $t\in\T$, then
$$
	g(t-0) = a(t-0) - f(t) \neq a(t+0) - f(t) = g(t+0),
$$
which contradicts the fact that $g$ cannot have jump discontinuities. 
\end{proof}

\begin{remark}
We can prove the preceding theorem without the consequence of Janson's result and instead use the following Sarason's decomposition
$$
	VMO = \{ u + Pv : u,v\in C\}
$$
and the fact that the Riesz projection $P$ is bounded from $C$ onto $VMO$ (see, e.g., 1.48 of \cite{BS}). Now if $f \in C+H^\infty$, then clearly $f\in L^\infty$ and $Qf \in VMO$; and conversely, if $f\in L^\infty$ with $Qf\in VMO$, then $Qf = u+Pv$ for $u,v\in C$, and $Qf=Qu$, which implies $f-u\in H^\infty$, which gives the desired decomposition $f = u + (f-u) \in C + H^\infty$. Thus,
$$
	\{ f\in L^\infty : Qf \in VMO \} = C+H^\infty.
$$
It remains to observe that if $f\in BMO_{\log}$ and  $$\|f\|_{BMO_{\log}}=\sup_I  \frac{\log \frac{4\pi}{|I|}}{|I|}\left( \int_I |f-f_I| \right)$$ is the $BMO_{\log}-$seminorm of $f,$ then
$$
	\frac1{|I|} \int_I |f-f_I| = \frac{\log \frac{4\pi}{|I|}}{|I|}\left( \int_I |f-f_I| \right) \frac1{\log\frac{4\pi}{|I|}}
	\leq \frac{\|f\|_{BMO_{\log}}}{\log\frac{4\pi}{|I|}} \to 0
$$
as $|I|\to 0$, and so $BMO_{\log} \subset VMO$. Therefore, in the proof of Theorem~\ref{main}, we get $a = f+g$ for some $f\in C$ and $g\in H^\infty$ without Janson's result.
\end{remark}

We write $BMOA = BMO\cap H^1 = \{f\in BMO : f_k = 0\ {\rm for}\ k<0\}$. Fefferman's duality result $(H^1)^* = BMOA$ implies the following consequence of the previous theorem.

\begin{corollary}
If $a$ is the symbol of a bounded Toeplitz operator on $BMOA$, then $a$ cannot have jump discontinuities.
\end{corollary}
 
Our conclusion is that Fredholm theory for Toeplitz operators with piecewise continuous symbols cannot unfortunately be extended to the Hardy space $H^1$ (or to $BMOA$) within the context of the Banach algebra of bounded linear operators. It may be possible to consider this question in the framework of unbounded Fredholm operators. 

\textbf{Acknowledgments}. The authors thank Albrecht B\"ottcher for useful remarks.

\end{document}